\documentclass[12pt,a4paper]{article}
\usepackage{amsfonts}

\newtheorem{theorem}{Theorem}[section]

\newtheorem{corollary}{Corollary}[section]

\newtheorem{definition}{Definition}[section]
\newtheorem{example}{Example}[section]

\newtheorem{lemma}{Lemma}[section]

\newtheorem{proposition}{Proposition}[section]
\newtheorem{remark}{Remark}[section]

\newenvironment{proof}[1][Proof]{\noindent\textbf{#1.} }{\ \rule{0.5em}{0.5em}}
\input{tcilatex}
\begin{document}

\title{Resolvent metrics and heat kernel estimates}
\author{Andr\'{a}s Telcs \\
{\small Department of Computer Science and Information Theory, }\\
{\small Budapest University of Technology and Economics}\\
{\small Magyar tud\'{o}sok k\H{o}r\'{u}tja 2}\\
{\small H-1117 Budapest,}\\
{\small HUNGARY}\\
{\small \ telcs.szit.bme@gmail.com}}
\maketitle

\begin{abstract}
Resolvent metrics are generalization of the resistance metric and provide
unified treatment of heat kernel estimates of sub-Gaussian type under
minimal conditions.\ 
\end{abstract}

\tableofcontents

\section{Introduction}

Heat propagation is not only interesting on its own, but reflects the very
intrinsic structure of the space where it does take place. We gained such a
new insight by Kigami's resistance metric\cite{K1}. \ Unfortunately the use
of resistance metric is applicable only on recurrent spaces. \ In this paper
we eliminate this limitation and extend his notion to transient spaces, in
particular for a class of transient graphs.

In the course of study of heat propagation the analogy between results on
continuous and discrete spaces is utilized (see e.g. \cite{BBK}) and
switching between them become a powerful tool of the studies. That is why we
think that it is useful if we tackle, in the present paper, the technically
less involved random walk case and return to the measure metric space
version in a forthcoming paper.

Kigami's work and several other papers inspire the following questions.

For any given measure (Dirichlet) space is there a "good" metric in which:

\begin{itemize}
\item the elliptic Harnack inequality holds

\item a parabolic Harnack inequality holds (in conjunction with heat kernel
estimates)?
\end{itemize}

The presented results contribute to the answer of these questions.

We introduce the resolvent metric which is direct generalization of the
resistance metric and we make the following observations.\ 

\begin{itemize}
\item Under the resolvent metric the volume doubling property implies the
elliptic Harnack inequality.

\item Under the resolvent metric volume doubling property turns to be
equivalent to the parabolic Harnack inequality and two-sided heat kernel
estimates in a fully local sense (c.f. \cite{ln}).
\end{itemize}

\ The main results are given in Theorem \ref{tDUE},\ref{tUE},\ref{tNDLE} and %
\ref{tPH}. \ The paper concludes with examples.

The main result of the paper can be summarized as follows. \ We consider a
weighted graph $\left( \Gamma ,\mu \right) $ and a random walk on it. \ We
assume that for all $\mu _{x,y}>0$ we have $P\left( x,y\right) \geq p_{0}>0$
uniformly. \ We construct the resolvent metric $\rho $ and consider $B_{\rho
}\left( x,r\right) $, balls in $\rho ,$ their volume $V_{\rho }\left(
x,r\right) $ and define the scaling function $F\left( x,r\right) =\left(
r^{2}V_{\rho }\left( x,r\right) \right) ^{1/m}$ for a well chosen $m$. \
Denote $f\left( x,.\right) =F^{-1}\left( x,.\right) $ and $\widetilde{p}%
_{n}\left( x,y\right) =p_{n}\left( x,y\right) +p_{n+1}\left( x,y\right) $
the sum of the transition kernel. \ 

\begin{definition}
We define a set $W_{0}$ of scaling functions$,$ $F:\Gamma \times \left[
0,\infty \right] \rightarrow \mathbb{R}^{+}:$there is a $C>0$ such that for
all $x\in \Gamma ,r>0$%
\[
\frac{F\left( x,2r\right) }{F\left( x,r\right) }\leq C. 
\]
\end{definition}

\begin{theorem}
Volume doubling holds ($V_{\rho }\in W_{0}$) for $\mu $ with respect to $%
\rho $ if and only if there are $C>c>0,\beta >1,\delta >0$ and an $F\in
W_{0} $ such that for all $x,y\in \Gamma $ and $n>0$%
\[
p_{n}\left( x,y\right) \leq \frac{C}{V_{\rho }\left( x,f\left( x,n\right)
\right) }\exp \left[ -\left( \frac{F\left( x,r\right) }{n}\right) ^{\frac{1}{%
\beta -1}}\right] 
\]%
and for $\rho \left( x,y\right) \leq \delta f\left( x,n\right) $%
\[
\widetilde{p}_{n}\left( x,y\right) \geq \frac{C}{V_{\rho }\left( x,f\left(
x,n\right) \right) } 
\]%
hold.
\end{theorem}

\section{Basic definitions}

We consider $\left( \Gamma ,\mu \right) ,$ weighted graph, $\Gamma $ is a
countable infinite set of vertexes and $\mu _{x,y}=\mu _{y,x}\geq 0$ a
symmetric weight. \ Edges are formed \ by the pairs for which $\mu _{x,y}>0.$
We assume that the graph is connected.\ These weights define a measure on
vertices:%
\[
\mu \left( x\right) =\sum_{y\in \Gamma }\mu _{x,y} 
\]%
and on sets $A\subset \Gamma $%
\[
\mu \left( A\right) =\sum_{z\in A}\mu \left( z\right) . 
\]%
Due to the connectedness $\mu \left( x\right) >0$ for all $x$. \ \ It is
natural to define the random walk on weighted graphs, which is a reversible
Markov chain given by the one-step transition probabilities: 
\[
P\left( x,y\right) =\frac{\mu _{x,y}}{\mu \left( x\right) }. 
\]%
In what follows we always assume the condition $\left( p_{0}\right) $: there
is a constant $p_{0}>$ $0$ such that for all $x,y$ with $\mu _{x,y}>0$ 
\[
P\left( x,y\right) \geq p_{0} 
\]%
holds. \ One can define the transition operator $P$ on $c_{0}\left( \Gamma
\right) $ functions $Pf\left( x\right) =\sum P\left( x,y\right) f\left(
y\right) .$ \ The inner product for $c_{0}\left( \Gamma ,\mu \right) $ is
defined by $\left( f,g\right) =\left( f,g\right) _{c_{0}\left( \Gamma
\right) ,\mu }=\sum_{x}f\left( x\right) g\left( x\right) \mu \left( x\right)
.$

If $\rho $ is a metric, balls are defined with respect to it by 
\[
\widehat{B}_{\rho }\left( x,r\right) =\left\{ y:\rho \left( x,y\right)
<r\right\} 
\]%
Denote $B=B_{\rho }\left( x,r\right) $ the connected component of $\widehat{B%
}_{\rho }\left( x,r\right) $ containing $x.$ The volume of the connected
part $B$ is denote by $V_{\rho }\left( x,r\right) =\mu \left( B_{\rho
}\left( x,r\right) \right) $.

\begin{definition}
We say that the volume doubling property, $\left( VD\right) _{\rho }$ holds
if there is a $C_{\rho }>0$ constant such that for all $x\in \Gamma ,r>0$ 
\[
\frac{V_{\rho }\left( x,2r\right) }{V_{\rho }\left( x,r\right) }\leq C_{\rho
}. 
\]
\end{definition}

\section{The resolvent metric}

In case of recurrent spaces Kigami's observation is\ that the effective
resistance between two vertices $R\left( x,y\right) $ is metric. The
existence of the resistance metric has a particular consequence that, for
any $f$ in the domain of the Dirichlet form $\mathcal{E}$ 
\begin{equation}
\left\vert f\left( x\right) -f\left( y\right) \right\vert ^{2}\leq R\left(
x,y\right) \mathcal{E}\left( f,f\right) .  \label{key1}
\end{equation}%
If the volume of balls $V_{R}\left( x,r\right) $ with respect to the metric $%
R$ satisfies the doubling \ condition the following important estimate holds%
\[
R\left( x,B_{R}^{c}\left( x,r\right) \right) \asymp r. 
\]%
In particular 
\begin{equation}
R\left( x,B_{R}^{c}\left( x,r\right) \right) \geq cr  \label{rr>r}
\end{equation}%
while $R\left( x,B_{R}^{c}\left( x,r\right) \right) \leq r$ is evident. \
One may recognize that $\left( \ref{rr>r}\right) $ holds only for recurrent
weighted graphs. This is a nice particular situation which has been
successfully utilized in several papers to obtain heat kernel estimates and
stability results (\cite{BCK},\cite{K1}\cite{K2}). \ Almost the same proof
which leads to $\left( \ref{rr>r}\right) $ results the validity of the
Einstein relation in the form:%
\[
E_{x}\left( T_{B_{R}\left( x,r\right) }\right) \asymp rV_{R}\left(
x,r\right) 
\]%
and that the elliptic Harnack inequality follows from the bounded covering
condition. Having all that the heat kernel estimates and the parabolic
Harnack inequality follows.

The crucial observations fail in the transient case, first of all $\left( %
\ref{rr>r}\right) $ obviously does not hold, since $R\left(
x,B_{R}^{c}\left( x,r\right) \right) \rightarrow R_{0}>0$ and the rate of
convergence can be understand from the decay of $R\left( B_{R}\left(
x,r\right) ,B_{R}^{c}\left( x,2r\right) \right) $. \ 

\bigskip

In several previous works resolvents are used with success to analyze
transient walks and diffusions. \ The simplest resolvent is the following:%
\[
\sum_{n=0}^{\infty }n^{m}P_{n}\left( x,y\right) . 
\]%
It is clear that it is monotonically increasing in $m$ and may be infinite
if $P_{n}$ decays polynomially. \ In probabilistic terms one may consider
this sum as the average visit time of $y$ by the increasing family of
independent walkers which has $n^{m}$ members at time $n$. \ In independent
walkers we mean here that on a given site some new walkers "born" (according
to the expansion of the family tree) and start independent walk. \ In what
follows we need a modified version of the resolvent which provides nice
correspondence to the power of the Laplace operator while it has basically
the same propertyies. \ We fix an $m\in \mathbb{N}$ which will be specified
later and reserved as the parameter of the resolvent.

In \cite{GT2} we started the utilization of polyharmonic functions, Green
function as well as Green operators (or resolvents). \ Now we follow this
direction and find a new metric for non strongly recurrent graphs (weakly
recurrent and transient) which posses nice features.

Denote $P$ the transition operator on $l_{1}\left( \Gamma \right) $. \ $%
Pf\left( x\right) =\sum_{y\sim x}P\left( x,y\right) f\left( y\right) $.$\,$

\begin{definition}
The Laplace operator is defined as $\Delta =P-I.$ The Dirichlet form
corresponding to the Laplace operator is given by 
\begin{eqnarray*}
\mathcal{E}\left( f,g\right) &=&\mathcal{E}_{1}\left( f,g\right) =\left(
-\Delta f,g\right) =\left( \left( I-P\right) f,g\right) = \\
&=&\frac{1}{2}\sum_{x,y}\left( f\left( x\right) -f\left( y\right) \right)
\left( g\left( x\right) -g\left( y\right) \right) \mu _{x,y}
\end{eqnarray*}
\end{definition}

\begin{definition}
For $A,B\subset \Gamma ,A\cap B=\emptyset $ we define the \ resistance 
\[
R\left( A,B\right) =\inf_{f\in c_{0}}\left[ \mathcal{E}\left( f,f\right)
:f|_{A}=1,f|_{B}=0\right] ^{-1} 
\]
\end{definition}

Let $A\subset \Gamma ,\Delta ^{A}=P^{A}-I,\left( \Delta ^{A}\right)
^{m}=\left( -1\right) ^{m}\left( I-P^{A}\right) ^{m}$ the $m$-th iteration
of the Laplace operator for $m\geq 1$ integer. \ If $A=\Gamma $ we drop it
from the notation.

Let us recall, that $\lambda \left( A\right) =\inf_{f\neq 0}\frac{\mathcal{E}%
^{A}\left( f,f\right) }{\left\Vert f\right\Vert ^{2}}:\left\Vert
f\right\Vert ^{2}=\left\Vert f\right\Vert _{l^{2}\left( \Gamma ,\mu \right)
}^{2}.$

The domain of the Dirichlet form on $\Gamma $ is defined by $\mathcal{F}_{m}=%
\mathcal{F}^{A}\left( \mathcal{E}_{m}\right) =\left\{ f\in l_{2}\left(
\Gamma ,\mu \right) ,\mathcal{E}_{m}^{A}\left( f,f\right) >0\right\} $,
where the bilinear form $\mathcal{E}_{m}^{A}$ is defined as%
\[
\mathcal{E}_{m}^{A}\left( f,g\right) =\left( \left( I-P^{A}\right)
^{m}f,g\right) _{l_{2}\left( A,\mu \right) }. 
\]%
The quasi resolvent metric on $A$ is defined as%
\[
R_{m}^{A}\left( x,y\right) =\sup_{f}\left\{ \frac{\left\vert f\left(
x\right) -f\left( y\right) \right\vert ^{2}}{\mathcal{E}_{m}^{A}\left(
f,f\right) }:f\left( x\right) \neq f\left( y\right) ,f\in \mathcal{F}%
_{m}\right\} . 
\]%
Note that $R_{m}^{A}$ is decreasing in $A$ since $\mathcal{E}_{m}^{A}$ is
increasing by definition, consequently $R_{m}$ 
\begin{equation}
R_{m}\left( x,y\right) =\sup_{f}\left\{ \frac{\left\vert f\left( x\right)
-f\left( y\right) \right\vert ^{2}}{\mathcal{E}_{m}\left( f,f\right) }%
:f\left( x\right) \neq f\left( y\right) ,f\in \mathcal{F}_{m}\right\} .
\label{rm}
\end{equation}%
is existing.

That has the equivalent forms%
\[
R_{m}\left( x,y\right) =\sup_{g}\left\{ \left\vert g\left( x\right) -g\left(
y\right) \right\vert ^{2}:0<\mathcal{E}_{m}\left( g,g\right) \leq 1\right\} 
\]

and

\[
R_{m}^{-1}\left( A,B\right) =\inf \left\{ \mathcal{E}_{m}\left( f,f\right)
:f\in \mathcal{F}_{m},\mathcal{E}_{m}\left( f,f\right) >0,f|_{A}\left(
x\right) =1,f|_{B}=0\right\} . 
\]
The former one can be seen using $g=\frac{f}{\sqrt{\mathcal{E}_{m}\left(
f,f\right) }}.$

\begin{lemma}
1. For any $f\in \mathcal{F}\left( \mathcal{E}_{m}\right) $%
\begin{equation}
\left\vert f\left( x\right) -f\left( y\right) \right\vert ^{2}\leq
R_{m}\left( x,y\right) \mathcal{E}_{m}\left( f,f\right) .  \label{key2}
\end{equation}%
2. If $\Gamma $ is connected and $A\cap B=\emptyset $ then 
\[
0<R_{m}^{-1}\left( A,B\right) <\infty . 
\]
\end{lemma}

\begin{proof}
The statements follow from the definition.
\end{proof}

\begin{lemma}
If $A\subset B\subset D\subset \Gamma \ $then%
\begin{equation}
R_{m}\left( A,B^{c}\right) \leq R_{m}\left( A,D^{c}\right)  \label{monot}
\end{equation}
\end{lemma}

\begin{proof}
By definition for $D^{c}\subset B^{c}$ 
\begin{eqnarray*}
R_{m}^{-1}\left( A,B^{c}\right) &=&\inf \left\{ \mathcal{E}_{m}\left(
f,f\right) :f|_{A}\left( x\right) =1,f|_{B^{c}}=0\right\} \\
&\geq &\inf \left\{ \mathcal{E}_{m}\left( f,f\right) :f|_{A}\left( x\right)
=1,f|_{D^{c}}=0\right\} \\
&=&R_{m}^{-1}\left( A,D^{c}\right) .
\end{eqnarray*}
\end{proof}

\begin{lemma}
$R_{m}\left( x,y\right) $ is a quasi metric:\newline
for any $x,y\in \Gamma ,$%
\begin{eqnarray}
R_{m}\left( x,y\right)  &=&R_{m}\left( y,x\right) ,  \label{rsim} \\
R_{m}\left( x,y\right)  &=&0\text{ if and only if }x=y  \label{rnull} \\
R_{m}\left( x,y\right)  &\leq &2\left( R_{m}\left( x,z\right) +R_{m}\left(
z,y\right) \right) .  \label{rtriang}
\end{eqnarray}
\end{lemma}

\begin{proof}
The first statement ensured by the definition. For the second see the end of
the proof of \cite{Ku1} Proposition 3.1. \ The weak triangular inequality
can be see as follows: 
\begin{eqnarray*}
R_{m}\left( x,y\right) &=&\sup_{g}\left\{ \left\vert g\left( x\right)
-g\left( y\right) \right\vert ^{2}:0<\mathcal{E}_{m}\left( g,g\right) \leq
1\right\} \\
&\leq &\sup_{g}\left\{ 2\left\vert g\left( x\right) -g\left( z\right)
\right\vert ^{2}+2\left\vert g\left( z\right) -g\left( y\right) \right\vert
^{2}:0<\mathcal{E}_{m}\left( g,g\right) \leq 1\right\} \\
&\leq &\sup_{g}\left\{ 2\left\vert g\left( x\right) -g\left( z\right)
\right\vert ^{2}:0<\mathcal{E}_{m}\left( g,g\right) \leq 1\right\} \\
&&+\sup_{g}\left\{ 2\left\vert g\left( z\right) -g\left( y\right)
\right\vert ^{2}:0<\mathcal{E}_{m}\left( g,g\right) \leq 1\right\} \\
&=&2\left( R_{m}\left( x,z\right) +R_{m}\left( z,y\right) \right)
\end{eqnarray*}
\end{proof}

The next result of Mac%
\'{}%
\i as and Segovia is essential in our work.

\begin{theorem}
(\cite{MS})\label{Lrr}If $X$ is a non-empty set and $d$ is a quasisymmetric
with constant $K$ :%
\[
d\left( x,y\right) \leq K\left( d\left( x,z\right) +d\left( z,y\right)
\right) 
\]%
then, there is a metric $\rho $, such that%
\[
d^{p}\left( x,y\right) \asymp \rho \left( x,y\right) 
\]%
with $p=\frac{1}{1+\log _{2}K}$.
\end{theorem}

\begin{corollary}
There is a metric $\rho $ and \thinspace $C>1>c>0$ such that for all $x,y\in
\Gamma $%
\begin{equation}
c\rho ^{2}\left( x,y\right) \leq R_{m}\left( x,y\right) \leq C\rho
^{2}\left( x,y\right)  \label{rrr}
\end{equation}
\end{corollary}

Based on this theorem we define balls with respect to $\rho $: $\widetilde{B}%
_{\rho }\left( x,r\right) =\left\{ y:\rho \left( x,y\right) <r\right\} $. As
in the case of the resistance metric it can be that the balls are not
connected. Let \ $B=B\left( x,r\right) =B_{\rho }\left( x,r\right) \subset 
\widetilde{B}_{\rho }\left( x,r\right) $ be the connected subset of $%
\widetilde{B}_{\rho }\left( x,r\right) $ containing $x.$ With the same
slight abuse of notation we shall use $B_{R}$ for the sets (balls) with
respect to the quasi-metric $R_{m}.$ \ One can immediately observe that $%
\left( \Gamma ,\rho ,\mu \right) $ satisfies volume doubling if and only if
\ $\left( \Gamma ,R_{m},\mu \right) $ does. \ In addition if the bounded
covering property holds with respect to one of $R_{m}$ or $\rho $ it holds
for the other as well and it follows from volume doubling (c.f.. \cite{ln}) .

\begin{definition}
The graph $\Gamma $ with metric $\sigma $ satisfies the bounded covering
condition if there is an integer $M>0$ such that for all $x\in \Gamma ,$ $%
r>0 $ the ball $B_{\sigma }\left( x,2r\right) $ can be covered at most $M$
balls of radius $r$.
\end{definition}

\begin{definition}
The $m$-resolvent is defined for an integer $m>0$ as follows. Let $%
Q_{m}\left( n\right) =\binom{n+m-1}{m-1},$ $A\subset \Gamma $ finite set and
for $x,y\in A$%
\[
G_{m}^{A}\left( x,y\right) =\sum_{n}Q_{m}\left( n\right) P_{n}\left(
x,y\right) 
\]%
the corresponding Green kernel is $g_{m}^{A}\left( x,y\right) =\frac{1}{\mu
\left( y\right) }G_{m}^{A}\left( x,y\right) $.
\end{definition}

The Green operators $G^{A}$ defined as usual. \ It is worth to observe
immediately, that for $m=0$ $G_{m}^{A}=I^{A}$ and for $m=1$ $%
G_{m}^{A}=G^{A}, $ the usual Green operator. For infinite $A$ the resolvent
operator may be unbounded and the Green function is $\infty .$ \ \ For
finite sets due to the transience of the Markov chain with Dirichlet
boundary, these objects are well-defined.

\begin{lemma}
The Dirichlet Green kernel $g_{m}^{A}\left( x,y\right) $ for finite $%
A\subset \Gamma $ is a reproducing kernel with respect to $\mathcal{E}_{m}.$
\end{lemma}

\begin{proof}
Let $u\in \mathcal{F}_{m},u|_{A^{c}}=0$%
\begin{eqnarray*}
\mathcal{E}_{m}\left( g_{m}^{A},u\right) &=&\left( \left( -\Delta \right)
^{m}G_{m}^{A}\frac{1}{\mu \left( .\right) },u\right) \\
&=&\left( \delta _{x}\frac{1}{\mu \left( .\right) },u\right) \\
&=&u\left( x\right) .
\end{eqnarray*}
\end{proof}

The next corollary is immediate.

\begin{corollary}
\[
\mathcal{E}_{m}\left( g_{m}^{A}\left( x,.\right) ,g_{m}^{A}\left( x,.\right)
\right) =g_{m}^{A}\left( x,x\right) 
\]
\end{corollary}

\begin{lemma}
The minimal value in the definition of $R_{m}\left( x,A^{c}\right) $ of $%
\mathcal{E}_{m}\left( f,f\right) $ is taken by $g\left( y\right) =\frac{1}{%
g_{m}^{A}\left( x,x\right) }g_{m}^{A}\left( x,y\right) $ and%
\begin{equation}
R_{m}\left( x,A^{c}\right) =g_{m}^{A}\left( x,x\right)  \label{rg}
\end{equation}
\end{lemma}

\begin{proof}
Let $h$ be an other function with $h\left( x\right) =1$ $h|_{A^{c}}=0$ then $%
d=h-g,h=d+g$%
\[
\mathcal{E}_{m}\left( h,h\right) =\mathcal{E}_{m}\left( g,g\right) +\mathcal{%
E}_{m}\left( d,d\right) +2\mathcal{E}_{m}\left( d,g\right) 
\]%
but $\mathcal{E}_{m}\left( d,d\right) \geq 0$ while $\mathcal{E}_{m}\left(
d,g\right) =cd\left( x\right) =0$.
\end{proof}

\begin{lemma}
Assume $\left( \Gamma ,\rho \right) $ has the bounded covering property (or $%
\left( VD\right) _{\rho }$ ), then 
\[
R_{m}\left( x,B_{\rho }^{c}\left( x,r\right) \right) >cr^{2} 
\]
\end{lemma}

\begin{proof}
\ From $\left( \ref{rrr}\right) $\ and $\left( \ref{monot}\right) $ we have
that there is a $c>0$ such for $s=cr^{2},B_{R}\left( x,s\right) \subset
B_{\rho }\left( x,r\right) $ 
\[
R_{m}\left( x,B_{\rho }^{c}\left( x,r\right) \right) \geq R_{m}\left(
x,B_{R}^{c}\left( x,s\right) \right) .
\]%
If 
\begin{equation}
R_{m}\left( x,B_{R}^{c}\left( x,s\right) \right) \geq cs  \label{RmLE}
\end{equation}%
we are ready. Now we prove $\left( \ref{RmLE}\right) $ following the steps
of \cite{Ku1}. \ Let $B=B_{R}\left( x,s\right) ,y,z\in B_{R}\left(
x,s\right) \,$ and $R_{m}\left( y,z\right) <\lambda s,\lambda \leq 1.$ Let
us fix a $c$ and a $z\in B$ with $c_{1}s<R_{m}\left( x,z\right) <s.$ \ We
consider $0\leq q_{z}\left( y\right) =\frac{g_{m}^{D}\left( x,y\right) }{%
g_{m}^{D}\left( x,x\right) }\leq 1$ $\ m$-harmonic function on $%
D=B\backslash \left\{ z\right\} $ \ with $q_{z}\left( x\right)
=1,q_{z}\left( z\right) =0.$ By definition and the reproducing property of
the Green kernel \ 
\[
\mathcal{E}_{m}\left( q_{z},q_{z}\right) =q_{z}=\frac{1}{R_{m}^{B}\left(
x,z\right) }\leq \frac{1}{R_{m}\left( x,z\right) }
\]%
where $R_{m}^{B}\left( x,z\right) $ denotes the resolvent metric within $B,$
while 
\begin{eqnarray*}
\left\vert q_{z}\left( y\right) \right\vert ^{2} &=&\left\vert q_{z}\left(
y\right) -q_{z}\left( z\right) \right\vert ^{2}\leq R_{m}\left( x,y\right) 
\mathcal{E}_{m}\left( q,q\right)  \\
&\leq &\frac{CR_{m}\left( y,z\right) }{R_{m}\left( x,z\right) }\leq \frac{%
C\left( \lambda s\right) ^{2}}{\left( c_{1}s\right) ^{2}}<\frac{1}{2}
\end{eqnarray*}%
if $\lambda =\lambda _{1}$ is chosen enough small. \ Note that volume
doubling implies bounded covering of $B_{R}\left( x,s\right) \backslash
B_{R}\left( x,c_{1}s\right) .$ \ Let $B_{R}\left( z_{i},\lambda _{1}s\right) 
$ the set of covering sets (via the covering with smaller $B_{\rho }$ balls: 
$B_{\rho }\left( z_{i},cr\right) \subset B_{R}\left( z_{i}.\lambda
_{1}s\right) $ balls with some extra increase of the covering number), $%
i=1,...,K.$ \ Denote $q\left( y\right) =\min_{i}q_{z_{i}}\left( y\right) $
and $q=2\left( q-\frac{1}{2}\right) ^{+}1_{B_{\rho }\left( x,r\right) }$. We
have that $q\left( x\right) =1$ and $q\left( y\right) =0$ on $%
B_{R}^{c}\left( x,s\right) $. Finally we obtain that%
\[
R_{m}^{-1}\left( x,B_{R}^{c}\left( x,s\right) \right) \leq \mathcal{E}%
_{m}\left( q,q\right) \leq 4\sum_{i}\mathcal{E}_{m}\left( q_{z}.q_{z}\right)
\leq \frac{4M}{\min_{i}R_{m}\left( x,z_{i}\right) }\leq \frac{4M}{c_{1}s}.
\]
\end{proof}

\begin{corollary}
\begin{equation}
R_{m}\left( x,B_{\rho }^{c}\left( x,r\right) \right) \asymp cr^{2}
\label{Rmr2}
\end{equation}
\end{corollary}

\begin{proof}
The lower estimate was given above, the upper one is almost immediate. We
chose $S=Cr^{2}$ so that $B_{R}\left( x,S\right) \supset B_{\rho }\left(
x,r\right) .$ Let $y\in \partial B_{R}\left( x,S\right) $ and apply $\left( %
\ref{monot}\right) $ for $\left\{ y\right\} \subset B_{R}^{c}\left(
x,S\right) =$ 
\[
R_{m}\left( x,B_{R}^{c}\left( x,S\right) \right) \leq R_{m}\left( x,y\right)
=S=Cr^{2}. 
\]
\end{proof}

\section{The tail distribution of the exit time}

This section contains two key results. \ One establishes an estimate similar
to the Einstein relation, the other presents the estimate of the tail
distribution of the exit time. \ The novelty in the approach is that in the
lack of the usual Einstein relation all the arguments should be accommodated
to the $m$--resolvent.

For brevity we will use the following notations:

$E_{m}\left( A|x\right) =\mathbb{E}_{x}\left( Q_{m+1}\left( T_{A}\right)
|X_{0}=x\right) $

$\overline{E}_{m}\left( A\right) =\max_{x\in A}E_{m}\left( A|x\right) $

$E_{m}\left( x,r\right) =E_{m}\left( x,r\right) =\mathbb{E}_{m}\left(
B_{\rho }\left( x,r\right) |x\right) $. We will use the particular notation
for $m=0,$

$E_{\rho }\left( x,r\right) =\mathbb{E}_{0}\left( B_{\rho }\left( x,r\right)
|x\right) $ is the usual mean exit time.

.

\begin{lemma}
\label{p<}For a set $A\subset \Gamma ,x\in A,$ there is a $C_{0}>1$ such that%
\[
\mathbb{P}_{x}\left( T_{A}<n\right) \leq 1-\frac{E_{m}\left( A\right) }{C%
\overline{E}_{m}\left( A\right) }+\frac{Cn^{m}}{\overline{E}_{m}\left(
A\right) } 
\]%
\newline
\end{lemma}

\begin{proof}
\begin{eqnarray*}
T_{A} &\leq &2n+I\left( T_{A}>n\right) T_{A}\circ \Theta _{n}, \\
T_{A}^{m} &\leq &2^{m}\left( \left( 2n\right) ^{m}+I\left( T_{A}>n\right)
T_{A}^{m}\circ \Theta _{n}\right) ,
\end{eqnarray*}%
where $\Theta _{n}$ is the time shift operator. From the strong Markov
property one obtains with $C=2^{\left\lceil m\right\rceil }$%
\begin{eqnarray*}
E_{m}\left( A\right) &\leq &C^{2}n^{m}+C\mathbb{E}_{x}\left( I\left(
T_{A}>n\right) \mathbb{E}_{X_{n}}\left( T_{A}^{m}\right) \right) \\
&\leq &C^{2}n^{m}+C\mathbb{P}_{x}\left( T_{A}>n\right) \overline{E}%
_{m}\left( A\right) . \\
\frac{E_{m}\left( A\right) }{C\overline{E}_{m}\left( A\right) } &\leq &\frac{%
Cn^{m}}{\overline{E}_{m}\left( A\right) }+\mathbb{P}_{x}\left( T_{A}>n\right)
\end{eqnarray*}%
and the statement follows.
\end{proof}

Let us recall here that under $\left( VD\right) _{\rho }$ the scaling
function $H\left( x,r\right) =r^{2}V_{\rho }\left( x,r\right) $ has nice
regularity properties.

\begin{corollary}
If $\left( VD\right) _{\rho }$ holds then there is a $c_{0}$ such that if $%
n=\left( \frac{1}{2}C_{0}^{-2}E_{m}\left( x,r\right) \right) ^{1/m}$ 
\begin{equation}
\mathbb{P}_{x}\left( T_{B_{\rho }\left( x,r\right) }\geq n\right) \geq c_{0}.
\label{P>c}
\end{equation}%
Here $C_{0}$ is given by the Lemma \ref{p<}.
\end{corollary}

\begin{theorem}
\label{TP}If $\left( \Gamma ,\rho \right) $ satisfies $\left( VD\right)
_{\rho }$ then, for $B=B_{\rho }\left( x,r\right) $%
\begin{equation}
\mathbb{P}_{x}\left( T_{B_{\rho }\left( x,r\right) }<n\right) \leq C\exp
\left( -ck_{m}\left( x,n,r\right) \right)  \label{P}
\end{equation}%
where $k=k_{m}\left( x,n,r\right) >1$ is the maximal integer for which%
\begin{equation}
\frac{n^{m}}{k}\leq q\min_{y\in B_{\rho }\left( x,r\right) }E_{m}\left(
B_{\rho }\left( y,\frac{r}{k}\right) \right) ,  \label{kdef}
\end{equation}%
where $q$ is a small constant (to be specified later).
\end{theorem}

\begin{definition}
Let us define $\beta _{m}$ as the smallest possible exponent for which%
\begin{equation}
\frac{R^{2}V_{\rho }\left( x,R\right) }{r^{2}V_{\rho }\left( x,r\right) }%
\leq C\left( \frac{R}{r}\right) ^{\beta _{m}},  \label{vv}
\end{equation}%
and observe that $\left( \ref{vv}\right) $ equivalent to $\left( VD\right)
_{\rho }$.
\end{definition}

\begin{remark}
There are several further trancripts of $\left( \ref{P}\right) .$ \ In the
simplest case if $r^{2}V_{\rho }\left( x,r\right) \asymp r^{\beta
},B=B_{\rho }\left( x,r\right) $ one has%
\begin{equation}
\mathbb{P}_{x}\left( T_{B}<n\right) \leq C\exp \left( -c\left( \frac{%
r^{\beta }}{n^{m}}\right) ^{\frac{1}{\beta -1}}\right) .  \label{Pi beta}
\end{equation}
\end{remark}

\begin{remark}
From $\left( \ref{Pi beta}\right) $ one can see that the estimate is weaker
as $m$ increases. If a lower estimate of the same form and magnitude is
aimed, $m$ should be chosen as small as possible. However it should be
recognized, that the increase of $m$ not only increase the upper bound but
the probability on the left hand side of $\left( \ref{P}\right) $.
\end{remark}

\begin{proof}[Proof of Theorem \protect\ref{TP}]
The proof follows the old, nice idea of \cite{BB1} (see also \cite{B1} Lemma
3.14). \ The only modification is that we use the very rough estimate:%
\[
T_{B_{\rho }\left( x,r\right) }^{m}\geq \sum_{i=1}^{k}\tau _{i}^{m} 
\]%
where $\tau _{i}$ is the exit time of $\partial B_{\rho }\left( \xi _{i},%
\frac{r}{k}\right) ,\xi _{i}=X_{\tau _{i-1}}$ and $k\geq 1$ will be chosen
later. \ \ \ From Lemma $\ref{p<}$ we have that with $t=\frac{n}{k}$%
\begin{equation}
P\left( \tau <t\right) \leq p+at^{m}  \label{ttt}
\end{equation}%
where $p\in \left[ \frac{1}{2},1-\varepsilon \right] $ and $a=\frac{2^{m}}{%
\overline{E}_{m}\left( x,\frac{r}{k}\right) }.$ \ Let $\eta $ be such that $%
P\left( \tau <t\right) =\left( p+at^{m}\right) \wedge 1$. \ The relation $%
\left( \ref{ttt}\right) can$ be rewritten as%
\[
P\left( \tau ^{m}<s\right) \leq p+as 
\]%
\[
\mathbb{E}\left( \exp \left( -\lambda \tau ^{m}\right) \right) \leq \mathbb{E%
}\left( \exp \left( -\lambda \eta ^{m}\right) \right) \leq p+a\lambda ^{-1}. 
\]%
From that point the proof can be finished as in \cite{B1}.
\end{proof}

\subsection{The Einstein relation}

The relation between the mean exit time of a ball, its volume and resistance
is regarded as a key tool to obtain heat kernel estimates. \ In this section
we obtain the corresponding relation with respect to the distance $\rho $
assuming only $\ $volume doubling and existence of the $m\geq 0$ integer. \
More precisely we show the following statements.

\begin{theorem}
\label{TER}If $\left( \Gamma ,\mu ,\rho \right) $ satisfies $\left(
VD\right) _{\rho }\ $then, satisfies $\left( ER\right) _{\rho }:$%
\begin{equation}
E_{\rho }\left( x,2r\right) \asymp \left[ R_{m}\left( x,B^{c}\right) V_{\rho
}\left( x,2r\right) \right] ^{1/m}  \label{ERrr}
\end{equation}%
with $B=B_{\rho }\left( x,2r\right) ,E_{\rho }\left( x,r\right) =\mathbb{E}%
_{m}\left( B_{\rho }\left( x,r\right) |x\right) .$
\end{theorem}

Theorem \ref{TER} will follow from the next statement and from the tail
estimate $\left( \ref{rPH}\right) $ of the exit time.

\begin{theorem}
\label{TERm}If $\mu $ satisfies $\left( VD\right) _{\rho }$ then, $\left(
ER\right) _{m}:$%
\begin{equation}
E_{m}\left( x,2r\right) \asymp R_{m}\left( x,2r\right) V_{\rho }\left(
x,2r\right)  \label{ERm}
\end{equation}%
holds, where $B=B_{\rho }\left( x,r\right) .$ $E_{m}\left( x,r\right) =%
\mathbb{E}_{m}\left( B|x\right) $,$R_{m}\left( x,2r\right) =R_{m}\left(
x,B_{\rho }^{c}\left( x,2r\right) \right) $
\end{theorem}

Let us recall, that $\mathbb{E}_{m}\left( B|x\right) =\mathbb{E}\left(
Q_{m+1}\left( T_{B}\right) |X_{0}=x\right) $ and Lemma 8.4 from\cite{ln}.

The first lemma is elementary.

\begin{lemma}
\label{Lmmm}Let $B=B_{\rho }\left( x,r\right) ,T=T_{B}$ then%
\[
\mathbb{E}_{x}\left( Q_{m+1}\left( T\right) \right) \asymp \mathbb{E}%
_{x}\left( T^{m}\right) . 
\]
\end{lemma}

\begin{proof}
Let $T=T_{B_{\rho }\left( x,r\right) }.$ \ Assume that $r$ is large enough
to ensure $T>2m$ i.e. $B_{d}\left( x,2m+1\right) \subset B_{\rho }\left(
x,r\right) $ and obtain%
\[
\frac{\left( 2T\right) ^{m}}{m!}\geq \frac{\left( T+m\right) ^{m}}{m!}\geq
Q_{m+1}\left( T\right) \geq \frac{\left( T-m\right) ^{m}}{m!}\geq c\left( 
\frac{T-m}{m}\right) ^{m}\geq \left( \frac{T}{2m}\right) ^{m}. 
\]%
For small values the inequality follows by adjusting the constants. \ 
\end{proof}

Of course the statement holds for arbitrary finite set as well.

\begin{remark}
\ From the definitions, the Theorem \ref{TP}, $\left( VD\right) _{\rho }$
and $\left( ER\right) _{\rho }$ it is immediate that 
\begin{eqnarray}
\frac{n^{m}}{k+1} &\geq &q\min_{y\in B_{\rho }\left( x,r\right) }E_{m}\left(
B_{\rho }\left( y,\frac{r}{k}\right) \right) \\
&\geq &cq\min_{y\in B_{\rho }\left( x,r\right) }E_{m}\left( B_{\rho }\left(
y,r\right) \right) k^{-\beta _{m}}, \\
\left( k+1\right) ^{\beta _{m}-1} &\geq &cq\frac{E_{m}\left( B_{\rho }\left( 
\underline{y},r\right) \right) }{n^{m}} \\
k+1 &\geq &c\left( \frac{E_{m}\left( B\right) }{n^{m}}\right) ^{\frac{1}{%
\beta _{m}-1}},
\end{eqnarray}%
where $B=B\left( x,r\right) ,$ which yields%
\begin{eqnarray}
\mathbb{P}_{x}\left( T_{B}<n\right) &\leq &C\exp \left( -c\left( \frac{%
E_{m}\left( B\right) }{n^{m}}\right) ^{\frac{1}{\beta _{m}-1}}\right)
\label{rPE} \\
\mathbb{P}_{x}\left( T_{B}<n\right) &\leq &C\exp \left( -c\left( \frac{%
H\left( x,r\right) }{n^{m}}\right) ^{\frac{1}{\beta _{m}-1}}\right)
\label{rPH}
\end{eqnarray}
\end{remark}

\begin{lemma}
(Feynmann-Kacc formula, c.f.. \cite{GT2} or \cite{ln}) \ Let $f$ be a
function on $\Gamma ,A\subset \Gamma ,\lambda >0$ satisfying 
\[
\Delta f-\lambda f=0\text{ in }B. 
\]%
Then for any $x\in A,\omega =\left( 1+\lambda \right) ^{-1},T=T_{A}$%
\[
f\left( x\right) =\mathbb{E}_{x}\left[ \omega ^{T}f\left( X_{T}\right) %
\right] 
\]%
and for any $m\geq 0$%
\begin{equation}
\left( \left[ \sum \omega ^{n+m}Q_{m}\left( n\right) P_{n}^{A}\right]
f\right) \left( x\right) =\mathbb{E}_{x}\left( Q_{m+1}\left( T\right) \omega
^{T+m}f\left( X_{T}\right) \right) .  \label{glm}
\end{equation}
\end{lemma}

\begin{corollary}
If we choose $f\equiv 1,\lambda =0$ we have from $\left( \ref{glm}\right) $
that%
\begin{equation}
E_{m}\left( A|x\right) =\mathbb{E}_{x}\left( Q_{m+1}\left( T_{A}\right)
\right) =\sum_{y\in B}G_{m}^{A}\left( x,y\right) .  \label{GB}
\end{equation}
\end{corollary}

\begin{lemma}
\label{LE<ggm}For any $A\subset \Gamma ,x\in A$%
\begin{equation}
E_{m}\left( A|x\right) \leq g_{m}^{A}\left( x,x\right) \mu \left( A\right) .
\label{E<gm}
\end{equation}
\end{lemma}

\begin{proof}
The proof follows from $\left( \ref{GB}\right) $.
\end{proof}

\begin{corollary}
\[
E_{m}\left( A|x\right) \leq R_{m}\left( x,A^{c}\right) \mu \left( A\right) 
\]%
and in particular for $x\in \Gamma ,r>0$%
\[
E_{m}\left( x,r\right) \leq Cr^{2}V_{\rho }\left( x,r\right) 
\]
\end{corollary}

Both statements direct consequence of $\left( \ref{E<gm}\right) $.

\begin{proof}
(of Theorem \ref{TERm}). \ The upper \ estimate is provided by Lemma \ref%
{LE<ggm}, for the lower estimate we apply the proof of Proposition 4.2 in 
\cite{Ku1}. Denote $B=B_{\rho }\left( x,2r\right) $. \ We start with $\left( %
\ref{key2}\right) $: If $f\in \mathcal{F}\left( \mathcal{E}_{m}\right) $%
\[
\left\vert f\left( x\right) -f\left( y\right) \right\vert ^{2}\leq
R_{m}\left( x,y\right) \mathcal{E}_{m}\left( f,f\right) 
\]%
in particular let $g\left( z\right) =g_{m}^{B}\left( x,z\right) $ and $z\in
B_{\rho }\left( x,\delta r\right) $ then 
\[
\left\vert g\left( x\right) -g\left( z\right) \right\vert ^{2}\leq
R_{m}\left( x,z\right) \mathcal{E}_{m}\left( g,g\right) . 
\]%
From reproducing property of $g_{m}^{B}\left( x,z\right) $ we have that $%
\mathcal{E}_{m}\left( g,g\right) =R_{m}\left( x,2r\right) =g_{m}^{B}\left(
x,x\right) \geq cr^{2}$%
\[
\left\vert g\left( x\right) -g\left( z\right) \right\vert ^{2}\leq C\delta
^{2}r^{2}=C\delta ^{2}\leq C\delta ^{2}g\left( x\right) ^{2}. 
\]%
We can choose $\delta $ such that $C\delta ^{2}\leq 2,$ and we obtain from $%
g\left( z\right) \leq g\left( x\right) ,$ that for $z\in B_{\rho }\left(
x,\delta r\right) $ \ 
\[
g_{m}^{B}\left( x,z\right) \geq \frac{1}{2}g_{m}^{B}\left( x,x\right) . 
\]%
Now we finish immediately using the definition and $\left( VD\right) _{\rho
} $.%
\begin{eqnarray*}
E_{m}\left( x,2r\right) &=&\sum_{y\in B}g_{m}^{B}\left( x,y\right) \mu
\left( y\right) \geq \sum_{z\in B_{\rho }\left( x,\delta r\right)
}g_{m}^{B}\left( x,z\right) \mu \left( z\right) \\
&\geq &\frac{1}{2}g_{m}^{B}\left( x,x\right) V_{\rho }\left( x,\delta
r\right) \\
&\geq &cR_{m}\left( B_{\rho }\left( x,r\right) ,B_{\rho }^{c}\left(
x,2r\right) \right) V_{\rho }\left( x,2r\right) ,
\end{eqnarray*}%
where the last step follows from $\left( \ref{rg}\right) $ and $\left(
VD\right) _{\rho }$.
\end{proof}

After the above preparations the proof of Theorem \ref{TER} is immediate
from Theorem \ref{TERm} and the next Lemma.

\begin{lemma}
\label{LEQTm} 
\[
E_{\rho }\left( x,r\right) \asymp E_{m}\left( x,r\right) ^{1/m} 
\]
\end{lemma}

\begin{proof}
Let $B=B_{\rho }\left( x,r\right) ,T=T_{B}$ From the Jensen inequality we
obtain that for $m\geq 1$%
\[
E_{m}\left( x,r\right) =\mathbb{E}_{x}\left( Q_{m+1}\left( T_{B_{R}\left(
x,r\right) }\right) \right) \asymp \mathbb{E}_{x}\left( T^{m}\right) \geq %
\left[ E_{\rho }\left( x,r\right) \right] ^{m}. 
\]%
For the opposite estimate denote $E=E_{m}\left( x,r\right) $ and 
\begin{eqnarray*}
E_{\rho }\left( x,r\right) &=&\sum_{n}P\left( T>n\right) \geq
\sum_{n=c_{0}\left( E\right) ^{1/m}}^{2c_{0}\left( E\right) ^{1/m}}P\left(
T>n\right) \\
&\geq &c_{0}E^{1/m}P\left( T_{B}>2c_{0}E^{1/m}\right) .
\end{eqnarray*}%
Now we use Theorem \ref{TP}, in particular $\left( \ref{rPH}\right) $ \ 
\begin{equation}
\mathbb{P}_{x}\left( T<n\right) \leq C\exp \left( -c\left( \frac{H\left(
x,r\right) }{n^{m}}\right) ^{\frac{1}{\beta _{m}-1}}\right)
\end{equation}%
\begin{eqnarray*}
&&\mathbb{P}_{x}\left( T\geq n\right) \\
&=&1-\mathbb{P}_{x}\left( T<n\right) \\
&\geq &1-C\exp \left( -c\left( \frac{H\left( x,r\right) }{n^{m}}\right) ^{%
\frac{1}{\beta -1}}\right) \geq 1/2
\end{eqnarray*}%
if we chose $n^{m}\geq H\left( x,r\right) $ and $c_{0}\,$\ such that $%
logC-c\left( \frac{1}{2c_{0}}\right) ^{\frac{1}{\beta -1}}=1/2$ i.e. $c_{0}=%
\frac{1}{2}\left( \frac{c}{\log C-1/2}\right) ^{\beta _{m}-1}$ the proof is
complete.
\end{proof}

\begin{proof}[Proof of Theorem $\left( \protect\ref{TER}\right) $]
The statement is immediate from Lemma \ref{Lmmm}, \ref{LEQTm} and Theorem %
\ref{TERm}.
\end{proof}

\section{Heat kernel estimates}

In this section we show that $\left( VD\right) _{\rho }$ implies the of
off-diagonal upper and near diagonal lower estimates. The proofs are
adaptation of the ones developed in the works \cite{BCK},\cite{Ku1} and \cite%
{K1} all, on resistance forms in case of recurrent (or strongly recurrent)
spaces, graphs. \ 

It is standard to deduce a diagonal lower estimate from $\left( \ref{P>c}%
\right) ,$ (see also in \cite{ln} Theorem 6.2 ).

\begin{proposition}
If $\left( VD\right) _{\rho }$ holds there is a $c>0$ such that for all $n>0$
\begin{equation}
p_{2n}\left( x,x\right) \geq \frac{c}{V_{\rho }\left( x,f\left( x,n\right)
\right) },  \tag{DLE}  \label{DLE}
\end{equation}%
where $f\left( x,.\right) $ is the inverse of $F\left( x,.\right) $ in the
second variable.
\end{proposition}

\begin{remark}
\ If we consider the example, $V_{\rho }\left( x,r\right) \asymp r^{\alpha }$
then $H\left( x,r\right) =r^{2}V_{\rho }\left( x,r\right) \asymp r^{\alpha
+2},$ $F\left( x,r\right) \asymp r^{\frac{\alpha +2}{m}},f\left( x,n\right)
\asymp n^{\frac{m}{\alpha +2}}$%
\[
p_{2n}\left( x,x\right) \geq cn^{-\frac{\alpha }{\left( \alpha +2\right) }%
m}. 
\]%
Thus 
\[
\sum_{n}n^{m}p_{n}\left( x,x\right) =\infty . 
\]
\end{remark}

\begin{remark}
One may object that $m\left( 1-\frac{\alpha }{\left( \alpha +2\right) }%
\right) >-1$ holds for all $m$, that seemingly contradicts to the initial
argument, which indicates that $m$ had to be chosen enough large. \ Let us
notice, that if $m$ is not enough large then the whole above argument is
meaningless, the resolvent operator $G_{m}$ is bounded and we can not
subtract the needed asymptotic information from it. Among others the notion $%
R_{m}\left( x,y\right) $ similarly to the usual resistance metric $R\left(
x,y\right) $ in \ the transient case is meaningless, furthermore the key
observations $\left( \ref{key1}\right) \ $\ and $\left( \ref{key2}\right) $
are not in our possession.
\end{remark}

\subsection{Estimates of higher order time derivatives of the heat kernel
and relation to $\mathbb{\mathcal{E}}_{m}$}

A fairly simple but powerful method is developed in the mentioned works (see
also \cite{BJKS}). The key observation is the following (see for the
simplified proof \cite{K1} Theorem 10.4). Without any further assumption for
any finite set $A\subset \Gamma $%
\[
p_{n}\left( x,x\right) \leq \frac{2\overline{R}\left( x,A\right) }{n}+\frac{%
\sqrt{2}}{\mu \left( A\right) }, 
\]%
where $\overline{R}(x,A)=\sup_{y\in A}R(x,y).$ \ We have the following
version of the statement.

\begin{proposition}
\label{ppA}There is a $C>0$ such that, for any finite set $A\subset \Gamma $
and $x\in A,n>0$ integer%
\[
p_{n}\left( x,x\right) \leq C\left( \frac{\overline{R}_{m}\left( x,A\right) 
}{n^{m}}+\frac{1}{\mu \left( A\right) }\right) , 
\]%
where $\overline{R}_{m}(x,A)=\sup_{y\in A}R_{m}(x,y).$
\end{proposition}

Before we prove the statement we show how one can obtain the diagonal upper
bound from it.

\begin{theorem}
\label{tDUE}Assume $\left( VD\right) _{\rho }$ then%
\begin{equation}
p_{n}\left( x,x\right) \leq \frac{C}{V_{\rho }\left( x,f\left( x,n\right)
\right) },  \label{DUE1}
\end{equation}%
where $f\left( x,n\right) $ is the inverse of $F\left( x,r\right) $ in the
second variable, furthermore%
\begin{equation}
p_{n}\left( x,y\right) \leq \frac{C}{\sqrt{V_{\rho }\left( x,f\left(
x,n\right) \right) V_{\rho }\left( y,f\left( y,n\right) \right) }}.
\label{PUE1}
\end{equation}
\end{theorem}

\begin{proof}
\ Let $A=B=B_{\rho }\left( x,r\right) $ and choose $r$ to have $\frac{C%
\overline{R}_{m}\left( x,B\right) }{n^{m}}=\frac{C^{\prime }}{V_{\rho
}\left( x,r\right) }$ $\ $and $n=\left\lceil Cr^{2}V_{\rho }\left(
x,r\right) \right\rceil ^{1/m}$ then%
\[
p_{2\left[ Cr^{2}V_{\rho }\left( x,r\right) \right] ^{1/m}}\left( x,x\right)
\leq \frac{C}{V_{\rho }\left( x,r\right) } 
\]%
\[
p_{2n}\left( x,x\right) \leq \frac{C}{V_{\rho }\left( x,f\left( x,n\right)
\right) }. 
\]%
This shows the statement for even $n,$ for odd $n$ it can be seen together
with \ $\left( \ref{PUE1}\right) $ as in \cite{GT2}.
\end{proof}

\begin{definition}
Let us introduce the time differential operator and its iterations for $%
k\geq 1,n\geq 0$%
\begin{eqnarray*}
\left( Df\right) _{n} &=&f_{n+2}-f_{n} \\
\left( D^{k}f\right) _{n} &=&\left( D\left( D^{k-1}f\right) \right) _{n} \\
f_{n}^{\left( k\right) } &=&\left( D^{k}f\right) _{n}.
\end{eqnarray*}
\end{definition}

\begin{lemma}
\label{pdiff}Let \ \ \ $f_{n}\left( y\right) =p_{n}\left( x,y\right) $ then
there is a $C>0$ such that for all $k\geq 0,n>0,x\in \Gamma $%
\begin{equation}
\left( -1\right) ^{k}f_{4n}^{\left( k\right) }\left( x\right) \leq \frac{1}{%
\left( 2n\right) ^{k}}f_{2n}\left( x\right)  \label{df}
\end{equation}
\end{lemma}

\begin{proof}
From the spectral decomposition of $p_{2n}\left( x,x\right) $ we know that $%
h_{2n}^{\left( k\right) }=\left( -1\right) ^{k}f_{n}^{\left( k\right)
}\left( x\right) \geq 0$ for all $k\geq 0$ and the same implies that the map 
$n\rightarrow \left( -1\right) \left( h_{2n+2}^{\left( k-1\right)
}-h_{2n}^{\left( k-1\right) }\right) \left( x\right) $ non-decreasing. \ We
show the statement by induction using a slightly stronger statement. \
Assume it holds for all $0\leq i<k,$ 
\[
h_{4n}^{\left( i\right) }\leq \frac{1}{\left( 4\left\lfloor
S_{i}n\right\rfloor \right) ^{i}}f_{4n-4\left\lfloor S_{i}n\right\rfloor
}\left( x\right) , 
\]
where $c_{i}=2^{-\left( i+2\right) }$ , $S_{k}=S_{k-1}+c_{k},$ and note that
for $i=0$ the assumption holds. \ 
\begin{eqnarray*}
h_{4n}^{\left( k\right) } &=&\left[ h_{4n}^{\left( k-1\right)
}-h_{4n+2}^{\left( k-1\right) }\right] \\
&\leq &\frac{1}{4\left\lfloor c_{k}n\right\rfloor }\sum_{i=0}^{4\left\lfloor
c_{k}n\right\rfloor }\left[ h_{4n-2i}^{\left( k-1\right)
}-h_{4n+2-2i}^{\left( k-1\right) }\right] \\
&\leq &\frac{1}{4\left\lfloor c_{k}n\right\rfloor }\left[ h_{4n-4\left%
\lfloor c_{k}n\right\rfloor }^{\left( k-1\right) }-h_{2n+2}^{\left(
k-1\right) }\right] \\
&\leq &\frac{1}{4\left\lfloor c_{k}n\right\rfloor }h_{4n-4\left\lfloor
c_{k}n\right\rfloor }^{\left( k-1\right) }
\end{eqnarray*}%
now by induction, if $m=n-\left\lfloor c_{k}n\right\rfloor $ 
\[
h_{4n}^{\left( k\right) }\leq \frac{1}{4m}h_{4m}^{\left( k-1\right) }\leq 
\frac{1}{4m}\frac{1}{\left( 4\left\lfloor S_{k-1}m\right\rfloor \right)
^{k-1}}f_{4m-4\left\lfloor S_{k-1}m\right\rfloor }\left( x\right) 
\]%
Let us recall that $f_{2k}\left( x\right) $ is non-increasing in $k$ and
find that%
\begin{eqnarray*}
4m-4\left\lfloor S_{k-1}m\right\rfloor &=&4\left( n-\left\lfloor
c_{k}n\right\rfloor \right) -4\left\lfloor S_{k-1}\left( n-\left\lfloor
c_{k}n\right\rfloor \right) \right\rfloor \\
&\geq &4\left( n-\left\lfloor c_{k}n\right\rfloor \right) -4\left\lfloor
S_{k-1}n\right\rfloor \\
&\geq &4n-4\left( \left\lfloor c_{k}n\right\rfloor +\left\lfloor
S_{k-1}n\right\rfloor \right) \\
&\geq &4n-4\left\lfloor \left( S_{k-1}+c_{k}\right) n\right\rfloor \\
&=&4n-4\left\lfloor S_{k}n\right\rfloor
\end{eqnarray*}%
which leads to the needed inequality. 
\[
h_{4n}^{\left( k\right) }\leq \frac{1}{\left( 4\left\lfloor
S_{k}n\right\rfloor \right) ^{k}}f_{4n-4\left\lfloor S_{k}n\right\rfloor
}\left( x\right) . 
\]%
Finally observing that $S_{k}=\sum_{i=0}^{k}2^{-\left( i+2\right) }$ we have
that $S_{k}<\frac{1}{2}$ and we obtain $\left( \ref{df}\right) $.
\end{proof}

\begin{lemma}
\label{Ledif}Again, if $f_{n}\left( y\right) =p_{n}\left( x,y\right) $, then%
\[
\mathcal{E}_{m}\left( f_{n},f_{n}\right) =\left( -1\right) ^{m}\left(
D^{m}f\right) _{2n}\left( x\right) =\left( -1\right) ^{m}f_{2n}^{\left(
m\right) }\left( x\right) 
\]
\end{lemma}

\begin{proof}
\[
\mathcal{E}_{m}\left( f_{n,}f_{n}\right) =\left( \left( I-P\right)
^{m}f_{n},f_{n}\right) 
\]%
but%
\begin{eqnarray*}
\left( \left( I-P\right) ^{m}f_{n},f_{n}\right) &=&\left(
\sum_{i=0}^{m}\left( -1\right) ^{i}\binom{m}{i}f_{n+i},f_{n}\right) \\
&=&\sum_{i=0}^{m}\left( -1\right) ^{i}\binom{m}{i}f_{2n+2i}\left( x\right) \\
&=&\left( -1\right) ^{m}\left( D^{m}f\right) _{2n}\left( x\right)
\end{eqnarray*}
\end{proof}

\begin{proof}
(of Proposition \ref{ppA}) Let $A\subset \Gamma $ be a finite set and choose 
$y^{\ast }$ so that 
\begin{eqnarray*}
p_{2n}\left( x,y^{\ast }\right) &:&=\min_{y\in A}p_{2n}\left( x,y\right) \\
p_{2n}\left( x,y^{\ast }\right) \sum_{z\in A}\mu \left( z\right) &\leq
&\sum_{z\in A}p_{2n}\left( x,z\right) \mu \left( z\right) \\
&\leq &\sum_{z\in \Gamma }P_{2n}\left( x,z\right) \leq 1,
\end{eqnarray*}%
and 
\[
p_{2n}\left( x,y^{\ast }\right) \leq \frac{1}{\mu \left( A\right) } 
\]%
follows. \ Let us denote $f_{n}\left( y\right) =p_{n}\left( x,y\right) $. \
By elementary estimates we have that

\begin{eqnarray*}
\frac{1}{2}p_{2n}^{2}\left( x,x\right) &\leq &p_{2n}^{2}\left( x,y^{\ast
}\right) +\left\vert p_{2n}\left( x,x\right) -p_{2n}\left( x,y^{\ast
}\right) \right\vert ^{2} \\
&\leq &\frac{1}{\mu ^{2}\left( A\right) }+\overline{R}_{m}\left( x,A\right) 
\mathcal{E}_{m}\left( f_{2n},f_{2n}\right) \\
&\leq &\frac{1}{\mu ^{2}\left( A\right) }+\overline{R}_{m}\left( x,A\right) 
\frac{C}{n^{m}}p_{2n}\left( x,x\right) ,
\end{eqnarray*}%
where in the last step Lemma \ref{pdiff} is used. \ Solving this for $%
p_{2n}\left( x,x\right) $ we obtain%
\begin{eqnarray}
p_{2n}\left( x,x\right) &\leq &C_{1}\frac{\overline{R}_{m}\left( x,A\right) 
}{t^{m}}+\left( \frac{2}{\mu ^{2}\left( A\right) }+C_{2}\frac{\overline{R}%
_{m}^{2}\left( x,A\right) }{t^{2m}}\right) ^{1/2} \\
&\leq &C\left( \frac{\overline{R}_{m}\left( x,A\right) }{n^{m}}+\frac{1}{\mu
\left( A\right) }\right) .
\end{eqnarray}%
The proof is finished by noting that $p_{2n+1}\left( x,x\right) \leq
p_{2n}\left( x,x\right) .$
\end{proof}

\subsection{\protect\bigskip The off-diagonal upper estimate}

The off-diagonal estimate can be easily obtained from the diagonal one. \ 

\begin{theorem}
\label{tUE}Assume $\left( p_{0}\right) ,\left( VD\right) _{\rho }$ and $%
\left( DUE\right) _{F}$ then 
\begin{eqnarray*}
p_{n}\left( x,y\right) &\leq &\frac{C}{V\left( x,f\left( x,n\right) \right) }%
\exp \left( -ck\left( x,n,r\right) \right) \\
&\leq &\frac{C}{V_{\rho }\left( x,f\left( x,n\right) \right) }\exp \left(
-c\left( \frac{F_{\rho }\left( x,d\left( x,y\right) \right) }{n}\right) ^{%
\frac{1}{\beta _{m}-1}}\right) .
\end{eqnarray*}%
The proof is word by word the same as for Theorem 8.5 in \cite{ln}\ or an
alternative proof is combination of Theorem 8.6 and 8.10 in \cite{ln}.
\end{theorem}

\subsection{Lower estimates}

The next task is to show the Near Diagonal Lower Estimate $\left(
NDLE\right) _{F}:$ There are $\delta $ and $c>0$ such that , for all $x\in
\Gamma ,r>0,y\in B\left( x,r\right) ,n>0$ if $\rho \left( x,y\right) \leq
\delta f\left( x,n\right) $ then, 
\[
\widetilde{p}_{n}\left( x,y\right) \geq \frac{c}{V_{\rho }\left( x,f\left(
x,n\right) \right) }. 
\]

\begin{theorem}
\label{tNDLE}If $\left( \Gamma ,\mu \right) $ satisfies $\left( VD\right)
_{\rho }$ and $\left( DUE\right) _{F}$ then, $\left( NDLE\right) _{F}$ holds.
\end{theorem}

\begin{proof}
First we prove 
\[
p_{2n}\left( x,y\right) \geq \frac{c}{V_{\rho }\left( x,f\left( x,n\right)
\right) }. 
\]
for $x,y\in \Gamma $ satisfying $d\left( x,y\right) \equiv 0$ $\func{mod}2$. 
$\ $Let us choose $\,r$ such that $n=F_{\rho }\left( x,r\right) =\left[
r^{2}V_{\rho }\left( x,r\right) \right] ^{1/m}$ and denote $f_{2n}\left(
y\right) =p_{2n}\left( x,y\right) ,$ then

\[
\left\vert f_{2n}\left( x\right) -f_{2n}\left( y\right) \right\vert ^{2}\leq
R_{m}\left( x,y\right) \mathcal{E}_{m}\left( f,f\right) . 
\]%
By Lemma \ref{Ledif} we have that 
\[
\left\vert f_{2n}\left( x\right) -f_{2n}\left( y\right) \right\vert ^{2}\leq
R_{m}\left( x,y\right) \mathcal{E}_{m}\left( f_{2n},f_{2n}\right)
=R_{m}\left( x,y\right) \left( -1\right) ^{m}f_{4n}^{\left( m\right) } 
\]%
and by Lemma \ref{pdiff} and the diagonal upper estimate%
\begin{eqnarray*}
\left\vert f_{2n}\left( x\right) -f_{2n}\left( y\right) \right\vert ^{2}
&\leq &R_{m}\left( x,y\right) \frac{1}{n^{m}}p_{2n}\left( x,x\right) \\
&\leq &\frac{\overline{R}_{m}\left( x,\delta r\right) }{\left[ F_{\rho
}\left( x,r\right) \right] ^{m}}p_{2n}\left( x,x\right) \\
&\leq &C\frac{\delta r^{2}}{r^{2}V_{\rho }\left( x,r\right) }p_{2n}\left(
x,x\right) \leq \frac{1}{4}p_{2n}^{2}\left( x,x\right) ,
\end{eqnarray*}%
if $\delta $ is small enough. \ The above inequality means that%
\[
p_{2n}\left( x,y\right) \geq \frac{1}{2}p_{2n}\left( x,x\right) \geq \frac{c%
}{V_{\rho }\left( x,r\right) }. 
\]%
Finally%
\begin{eqnarray*}
p_{2n+1}\left( x,y\right) &=&p\left( x,z\right) \mu \left( z\right)
p_{2n}\left( z,y\right) \\
&\geq &p_{0}p_{2n}\left( z,y\right) \\
&\geq &\frac{c}{V_{\rho }\left( x,r\right) }.
\end{eqnarray*}
\end{proof}

\section{\protect\bigskip Stability}

In this section we show that $\left( VD\right) _{\rho }$ implies the
parabolic Harnack inequality via the two-sided estimates. \ It is an
interesting by-product that in our scenarios the volume doubling property
implies the elliptic Harnack inequality.

It is already known that $\left( UE_{F}\right) $ and $\left( NDLE_{F}\right) 
$ are equivalent to the $F$-parabolic Harnack inequality (c.f.. \cite{ln})
where $F$ is properly regular function and both imply $\left( VD\right)
_{\rho }.$

\begin{definition}
The function class $W_{1}$ is defined as follows. \ $F\in W_{1}$ if there
are $\beta \geq \beta ^{\prime }>1,C\geq c>0$ such that for all $R>r>0,x\in
\Gamma ,y\in B\left( x,R\right) ,$%
\begin{equation}
c\left( \frac{R}{r}\right) ^{\beta ^{\prime }}\leq \frac{F\left( x,R\right) 
}{F\left( y,r\right) }\leq C\left( \frac{R}{r}\right) ^{\beta }.
\end{equation}
\end{definition}

\begin{definition}
\label{dpH}We say that $\left( \mathbf{PH}\right) _{F}$, the parabolic
Harnack inequality holds \ for a weighted graph $\left( \Gamma ,\mu \right) $
with respect to a function $F\in W_{1}$ if there is a constant $C>0$ such
that for any $x\in \Gamma ,R,k>0$ and any solution $u\geq 0$ of the heat
equation 
\[
\partial _{n}u=\Delta u 
\]%
on $\mathcal{D}=[k,k+F(x,R)]\times B(x,2R)$, the following is true. On
smaller cylinders defined by 
\begin{eqnarray*}
\mathcal{D}^{-} &=&[k+\frac{1}{4}F(x,R),k+\frac{1}{2}F(x,R)]\times B(x,R),%
\text{ } \\
\text{and }\mathcal{D}^{+} &=&[k+\frac{3}{4}F(x,R),k+F(x,R))\times B(x,R),
\end{eqnarray*}%
and taking $(n_{-},x_{-})\in \mathcal{D}^{-},(n_{+},x_{+})\in \mathcal{D}%
^{+},$%
\begin{equation}
d(x_{-},x_{+})\leq n_{+}-n_{-},  \label{smdist}
\end{equation}%
the inequality 
\[
u_{n_{-}}(x_{-})\leq C\widetilde{u}_{n_{+}}(x_{+}) 
\]%
holds, where the short notation $\widetilde{u}_{n}=u_{n}+u_{n+1}$ was used.
\end{definition}

\begin{remark}
At present it is not clear how the $\beta ^{\prime }>1$ condition follows
from the assumptions, while we expect it holds if $\mu $ is $\left(
VD\right) _{\rho }$.
\end{remark}

\begin{definition}
We say that the elliptic Harnack inequality, $\left( H\right) $ holds with
respect to $\mu ,$ and $\rho $ if there is a $C>0$ \ such that for all $x\in
\Gamma ,$ $r>0$ if $\ h$ is harmonic on $B_{\rho }\left( x,2r\right) :$%
\[
\left( I-P\right) h\left( x\right) =0\text{ for }x\in B\left( x,2r\right) 
\]%
then 
\[
\max_{z\in B\left( x,r\right) }h\left( z\right) \leq C\min_{y\in B\left(
x,r\right) }h\left( y\right) . 
\]
\end{definition}

\begin{theorem}
\label{tPH} Assume that $\left( \Gamma ,\mu \right) $ satisfies $\left(
p_{0}\right) $, then the following statements are equivalent.\newline
\newline
1. $\left( VD\right) _{\rho },$ and $F\in W_{1}$\newline
2. $\left( UE\right) _{F}$ \ and $\left( NDLE\right) _{F}$ hold for an $F\in
W_{1},$\newline
3. $\left( PH\right) _{F}$ holds for $F\in W_{1}$ with respect to $\rho $%
-balls.
\end{theorem}

\begin{proof}
The proof can be reproduced from the one of \cite{ln} Theorem 12.1.
\end{proof}

\begin{theorem}
\label{TH}Assume that $\left( \Gamma ,\mu \right) $ satisfies $\left(
p_{0}\right) $, \ then $\left( VD\right) _{\rho }$ and $F\in W_{1}$ implies $%
\left( H\right) .$
\end{theorem}

\begin{proof}
It is evident that $\left( PH\right) _{F}$ \ implies $\left( H\right) $ and
from Theorem \ref{tPH} we know that $\left( VD\right) _{\rho }$ and $F\in
W_{1}$ implies $\left( PH\right) _{F}$.\newline
\end{proof}

\begin{definition}
Two weighted graphs $\Gamma $ with $\mu $ and $\Gamma ^{\prime }$ with $\mu
^{\prime }$ are roughly isometric (or quasi isometric) with respect to the
metrics $d,d^{\prime }$ (c.f. \cite[Definition 5.9]{HK}) if there is a map $%
\phi $ from $\Gamma $ to $\Gamma ^{\prime }$ such that there are $a,b,c,M>0$
for which
\end{definition}

\begin{equation}
\frac{1}{a}d\left( x,y\right) -b\leq d^{\prime }\left( \phi \left( x\right)
,\phi \left( y\right) \right) \leq ad\left( x,y\right) +b  \label{r1}
\end{equation}%
for all $x,y\in \Gamma ,$

\begin{equation}
d^{\prime }\left( \phi \left( \Gamma \right) ,y^{\prime }\right) \leq M
\label{r2}
\end{equation}%
for all $y^{\prime }\in \Gamma ^{\prime }$ and 
\begin{equation}
\frac{1}{c}\mu \left( x\right) \leq \mu ^{\prime }\left( \phi \left(
x\right) \right) \leq c\mu \left( x\right)  \label{r3}
\end{equation}%
for all $x\in \Gamma .$

\begin{theorem}
\label{stab}The $F$-parabolic Harnack inequality is rough isometry invariant
with respect to $\rho $ and $\rho ^{\prime }$.
\end{theorem}

\begin{proof}
We know that $\left( VD\right) _{\rho }$ on $\left( \Gamma ,\mu \right) $ \
if and only if $\left( VD\right) _{\rho ^{\prime }}$ on $\left( \Gamma ,\mu
\right) .$ \ The equivalence of $\left( VD\right) _{\rho }$ and $\left(
PH\right) _{F}$ on both graphs are given by Theorem \ref{tPH} and the
statement follows.
\end{proof}

\subsection{Comments}

Kigami in \cite{K1} constructed a metric which is quasisymmetric to \ $%
RV_{R}\left( x,d\left( x,y\right) \right) +RV_{R}\left( x,d\left( y,x\right)
\right) $ . This procedure can be applied to $\rho ^{2}\left( x,y\right)
V_{\rho }\left( x,\rho \left( x,y\right) \right) +\rho ^{2}\left( x,y\right)
V_{\rho }\left( y,\rho \left( x,y\right) \right) $ \ as well. All the
conditions are satisfied to obtain a new metric $\sigma $ which is quasi
symmetric to $\rho $. \ We know that $\left( VD\right) _{\rho }$ implies $%
\left( VD\right) _{\sigma }$ and by Kigami's result \cite{K1} we have 
\begin{eqnarray*}
p_{n}\left( x,y\right) &\leq &\frac{C}{V_{\sigma }\left( x,g^{-1}\left(
n\right) \right) }\exp \left( -c\left( \frac{\sigma \left( x,y\right) }{n}%
\right) ^{\frac{1}{\beta -1}}\right) \\
\widetilde{p}_{n}\left( x,y\right) &\geq &\frac{c}{V_{\sigma }\left(
x,g^{-1}\left( n\right) \right) },
\end{eqnarray*}

where $g^{-1}\left( n\right) $ the inverse of $g\left( r\right)
=r^{a}V_{\sigma }\left( x,r\right) $ and $a$ is the exponent determined by
construction of $\sigma $ from $\rho .$ The lower estimate and parabolic
Harnack inequality follows for $\sigma $ as well. \ It should be noted here,
that with the introduction of the second new metric the dependence from $x$
in the exponential term is eliminated and $F\left( x,r\right) $ replaced by $%
g\left( r\right) $.

\section{Examples}

There are several examples of fractals and fractal like graph possessing
two-sided heat kernel estimates (see \cite{Bwhich}) and examples on which
direction dependence destroy it \cite{HK}. The major contribution of
resistance metric in showing heat kernel estimates is emphasized in \cite%
{BCK} and \cite{K1}. \ The same arguments apply to our work. \ 

\begin{example}
The graphs in Example 4. \cite{BCK}, are strongly recurrent. \ There, $\beta
>\alpha $ i.e. strong recurrence is assumed, that is not needed anymore. For
instance the high dimensional graphical Sierpinski carpet, can be handled by
our method.
\end{example}

Kigami constructed $\mathcal{G}$ a family of fractal structures in \cite{K1}%
. \ The structures can be discretized and get the so called graphical
Sierpinski gaskets, graphs. \ The typical structures are recurrent but it is
easy to lift up them and get transient graphs. \ In \cite{Bwhich} Barlow has
Proposition 5 as follows.

\begin{proposition}
Let $\alpha \geq 1$ and a graph $\left( \Gamma ,\mu \right) \in \mathcal{G}$
which satisfies $V\left( x,r\right) \simeq r^{\alpha }$ and $E\left(
x,r\right) \simeq r^{\beta }$ \ with respect to the shortest path metric,
furthermore the graph is very strongly recurrent (see the definition there).
\ Let $\lambda >0.$ \ Then there is a weighted graph $\left( \widetilde{%
\Gamma },\widetilde{\mu }\right) $ such that $V\left( x,r\right) \simeq
r^{\alpha +\lambda }$ but $E\left( x,r\right) \simeq r^{\beta }$ and
satisfies the elliptic Harnack inequality.
\end{proposition}

The graph $\left( \widetilde{\Gamma },\widetilde{\mu }\right) $ is product
of $\left( \Gamma ,\mu \right) $ and an ultrametric space. \ The choice of $%
\lambda >0$ can ensure that $\beta <\alpha +\lambda .$ \ From Proposition 3
of \cite{Bwhich} we know that such graphs are transient hence the resistance
metric is not applicable, while the resolvent metric does if volume doubling
holds for it. \ Polynomial volume growth in the resolvent metric follow from
the asymptotic spherical symmetry of the Green kernel and from polynomial
volume growth in the original metric.

\end{document}